\documentclass{birkjour}
 \newtheorem{thm}{Theorem}[section]
 
 \newtheorem{lem}[thm]{Lemma}
 
 \theoremstyle{definition}
 \newtheorem{defn}[thm]{Definition}
 \theoremstyle{remark}
 \newtheorem{rem}[thm]{Remark}

 \usepackage[mathscr]{eucal}
 \usepackage[OT1,T1,TS1,T2A]{fontenc}
\usepackage[cp1251]{inputenc}
\usepackage[russian,english]{babel}
 \numberwithin{equation}{section}

\newcommand{\taul}{\text{\LARGE\(\tau\)}}

\DeclareMathOperator{\tr}{tr}

\DeclareMathOperator{\im}{Im}
\DeclareMathOperator{\Mod}{(mod\, 2)}

\begin{document}

%
%
%
%
%

\title[On a family of Laurent polynomials]
 {On a family of Laurent polynomials\\ generated by \(\boldsymbol{2\times2}\) matrices}
\author[V. Katsnelson]{Victor Katsnelson}
\address{%
Department of Mathematics\\
The Weizmann Institute\\
76100, Rehovot\\
Israel}
\email{victor.katsnelson@weizmann.ac.il; victorkatsnelson@gmail.com}
\date{July, 2015}
\begin{abstract}
To \(2\times2\) matrix \(G\) with complex entries, the sequence of Laurent polynomial
\(L_n(z,G)=\tr \big(G\big[\begin{smallmatrix}z&0\\ 0&z^{-1}\end{smallmatrix}\big]G^{\ast}\big)^n\) is related. It turns out that for each \(n\), the family \(\big\{L_n(z,G)\big\}_G\), where \(G\) runs over
the set of all \(2\times2\) matrices, is a three-parametric family. A natural parametrization of this
family is found. The polynomial \(L_n(z,G)\) is expressed in terms of these parameters and the Chebyshev polynomial \(T_n\). The zero set of the polynomial \(L_n(z,G)\) is described.
\end{abstract}
\maketitle
\begin{itemize}
\item \(\mathbb{R}\) stands for the set of all real numbers.
\item  \(\mathbb{C}\) stands for the set of all complex numbers.
\item If \(z\in\mathbb{C},\,z=x+iy,\,x,y\in \mathbb{R}\), then \(\overline{z}=x-iy\)
is the complex conjugate number.
\item If \(M= \Big[\begin{smallmatrix}m_{11}&m_{12} \\[0.7ex] m_{21}&m_{22} \end{smallmatrix}\Big]\) is a matrix, then
    \(M^{\ast}=\Big[\begin{smallmatrix}\overline{m_{11}}&\overline{m_{21}} \\[0.7ex]
    \overline{m_{12}}&\overline{m_{22}} \end{smallmatrix}\Big]\)
    is the Hermitian conjugate matrix.
    \item For a matrix \(M= \Big[\begin{smallmatrix}m_{11}&m_{12} \\[0.7ex] m_{21}&m_{22} \end{smallmatrix}\Big]\), \(\tr M\) stands for the trace
        of  \(M\): \(\tr\,M=m_{11}+m_{22}\), \(\det M\) stands for the determinant
        of  \(M\).
\end{itemize}

 \section{Laurent polynomials generated by \(\boldsymbol{2\times2}\) matrices}
 \label{sec1}
 Let
  \begin{equation}
 \label{GM}
 G=
 \begin{bmatrix}
 g_{11}&g_{12}\\
 g_{21}&g_{22}
 \end{bmatrix}
 \end{equation}
be a \(2\times2\) matrix with complex entries. For \(z\in\mathbb{C}\), let us define
\begin{subequations}
\label{DePi}
\begin{gather}
\label{LM}
S(z,G)=G
\begin{bmatrix}
z&0\\ 0&z^{-1}
\end{bmatrix}
G^{\ast}, \\[1.0ex]
\label{DPi}
L_n(z,G)=\tr(S(z,G))^n.
\end{gather}
\end{subequations}
\begin{lem}
\label{let}
Considered as a function of \(z,\,\,L_n(z,G)\) is a Laurent polynomial:
\begin{equation}
\label{CoLp}
L_n(z,G)=\sum\limits_{-n\leq k\leq n}c_{k,n}(G)z^k.
\end{equation}
The "leading" coefficients \(c_{\pm n,n}(G)\) are:
\begin{equation}
\label{CoLpl}
c_{-n,n}(G)=(|g_{12}|^2+|g_{22}|^2)^n,\quad c_{n,n}(G)=(|g_{11}|^2+|g_{21}|^2)^n.
\end{equation}
\end{lem}
\begin{proof}
The \(2\times2\) matrix function \(S(z,G)\) can be presented as a linear combination:
\begin{equation}
S(z,G)=P_1(G)z+P_{-1}(G)z^{-1},
\end{equation}
where
\begin{equation}
P_1(G)=
\begin{bmatrix}
g_{11}\\
g_{21}
\end{bmatrix}
\begin{bmatrix}
\overline{g_{11}}&\overline{g_{21}}
\end{bmatrix},\quad
P_{-1}(G)=
\begin{bmatrix}
g_{12}\\
g_{22}
\end{bmatrix}
\begin{bmatrix}
\overline{g_{12}}&\overline{g_{22}}
\end{bmatrix}.
\end{equation}
Since \((S(z,G))^n=(P_{1}(G)z+P_{-1}(G)z^{-1})^n\), it is clear that
\begin{equation}
(S(z,G))^n=(P_1(G))^nz^n+\,\cdots\,+(P_{-1}(G))^nz^{-n}.
\end{equation}
The matrices \((P_1(G))^n,\,(P_{-1}(G))^n\) and their traces can be calculated easily:
\begin{gather}
\notag
(P_1(G))^n=\begin{bmatrix}
g_{11}\\
g_{21}
\end{bmatrix}
(|g_{11}|^2+|g_{21}|^2)^{n-1}
\begin{bmatrix}
\overline{g_{11}}&\overline{g_{21}}
\end{bmatrix},\\
\notag
(P_{-1}(G))^n=
\begin{bmatrix}
g_{12}\\
g_{22}
\end{bmatrix}
(|g_{12}|^2+|g_{22}|^2)^{n-1}
\begin{bmatrix}
\overline{g_{12}}&\overline{g_{22}}
\end{bmatrix},\\
\label{ETr}
\tr (P_1(G))^n=(|g_{11}|^2+|g_{21}|^2)^{n},\ \
\tr (P_{-1}(G))^n=(|g_{12}|^2+|g_{22}|^2)^{n}.
\end{gather}
Since the value \(\tr M\) is linear with respect to \(2\times2\) matrix \(M\),
the equality \eqref{CoLpl} follows from the definition \eqref{DPi} and from \eqref{ETr}.
\end{proof}
Let us "normalize" the polynom1al \(L_n(z,G)\).

\begin{defn}
\label{Gen}
We say that the matrix \(G=\big[\begin{smallmatrix}g_{11}&g_{12}\\ g_{21}&g_{22}\end{smallmatrix}\big]\) is \emph{generic} if the condition
\begin{equation}
(|g_{11}|^2+|g_{21}|^2)(|g_{12}|^2+|g_{22}|^2)\not=0
\end{equation}
holds.
\end{defn}
      For a generic matrix \(G=\begin{bmatrix}
 g_{11}&g_{12}\\
 g_{21}&g_{22}
 \end{bmatrix}\), \eqref{GM}, let us define
 \begin{gather}
 \label{Rs}
 R_1(G)=(|g_{11}|^2+|g_{21}|^2)^{\frac{1}{2}},\quad R_{2}(G)=(|g_{12}|^2+|g_{22}|^2)^{\frac{1}{2}},\\
 H=\begin{bmatrix}
 h_{11}&h_{12}\\
 h_{21}&h_{22}
 \end{bmatrix}=\begin{bmatrix}
 g_{11}&g_{12}\\
 g_{21}&g_{22}
 \end{bmatrix}\begin{bmatrix}(R_1(G))^{-1}&0\\ 0&(R_{2}(G))^{-1}\end{bmatrix}
 \label{HM}
 \end{gather}
 The matrix \(H\) satisfies the normalizing condition
 \begin{equation}
 \label{NoH}
 |h_{11}|^2+|h_{21}|^2=1, \ |h_{12}|^2+|h_{21}|^2=1.
 \end{equation}
 \begin{lem}
 \label{Rel}
 The Laurent polynomials \(L_n(\,.\,,H)\) and \(L_n(\,.\,,G)\) are related by the equality
 \begin{equation}
 \label{RelP}
 L_n(z,G)=R^{n}\,L_n(\rho\,z,H),
 \end{equation}
 where
 \begin{equation}
 \label{NoRa}
 R=R_{1}(G)R_2(G), \ \ \rho=R_1(G)/R_2(G),
 \end{equation}
   \(R_1(G),R_2(G)\) are defined by \eqref{Rs}.
 \end{lem}

\begin{lem}
\label{Mo}
Let \(H\) be an \emph{arbitrary} \(2\times2\) matrix with complex entries and \(F\) be the nonnegative square root of the matrix
\(H^{\ast}H\):
\begin{equation}
\label{SqR}
F^2=H^{\ast}H,\ \ F\geq0.
\end{equation}
Then the Laurent polynomials \(L_n(z,H)\) and \(L_n(z,F)\) coincides:
\begin{equation}
\label{CoP}
L_n(z,H)\equiv L_n(z,F).
\end{equation}
\end{lem}
\begin{proof} According to the definitions \eqref{LM} and \eqref{DPi},
\begin{equation}
L_n(z,H)=\tr\Big(H\big[\begin{smallmatrix}z&0\\0&z^{-1}\end{smallmatrix}\big]H^{\ast}\cdot
H\big[\begin{smallmatrix}z&0\\0&z^{-1}\end{smallmatrix}\big]H^{\ast}\cdot\,\,
\cdots\,\,\cdot H\big[\begin{smallmatrix}z&0\\0&z^{-1}\end{smallmatrix}\big]H^{\ast}\cdot
H\big[\begin{smallmatrix}z&0\\0&z^{-1}\end{smallmatrix}\big]H^{\ast}\Big)
\end{equation}
Permuting the matrices \(H\) and\\
\(\big[\begin{smallmatrix}z&0\\0&z^{-1}\end{smallmatrix}\big]H^{\ast}\cdot
H\big[\begin{smallmatrix}z&0\\0&z^{-1}\end{smallmatrix}\big]H^{\ast}\cdot
\cdots H\big[\begin{smallmatrix}z&0\\0&z^{-1}\end{smallmatrix}\big]H^{\ast}\cdot
H\big[\begin{smallmatrix}z&0\\0&z^{-1}\end{smallmatrix}\big]H^{\ast}\), we obtain
\begin{equation*}
L_n(z,H)=\tr\Big(\big[\begin{smallmatrix}z&0\\0&z^{-1}\end{smallmatrix}\big]H^{\ast}
H\cdot\big[\begin{smallmatrix}z&0\\0&z^{-1}\end{smallmatrix}\big]H^{\ast}
H\cdot\, \cdots\,\cdot\big[\begin{smallmatrix}z&0\\0&z^{-1}\end{smallmatrix}\big]H^{\ast}
H\cdot\big[\begin{smallmatrix}z&0\\0&z^{-1}\end{smallmatrix}\big]H^{\ast}H\big).
\end{equation*}
Taking into account \eqref{SqR}, we obtain
\begin{equation*}
L_n(z,H)=\tr\Big(\big[\begin{smallmatrix}z&0\\0&z^{-1}\end{smallmatrix}\big]F^2
\cdot\big[\begin{smallmatrix}z&0\\0&z^{-1}\end{smallmatrix}\big]F^2
\cdot\, \cdots\,\cdot\big[\begin{smallmatrix}z&0\\0&z^{-1}\end{smallmatrix}\big]F^2
\cdot\big[\begin{smallmatrix}z&0\\0&z^{-1}\end{smallmatrix}\big]F\cdot F\big).
\end{equation*}
Permuting the matrices
\(\big[\begin{smallmatrix}z&0\\0&z^{-1}\end{smallmatrix}\big]F^2
\cdot\big[\begin{smallmatrix}z&0\\0&z^{-1}\end{smallmatrix}\big]F^2
\cdot\, \cdots\,\cdot\big[\begin{smallmatrix}z&0\\0&z^{-1}\end{smallmatrix}\big]F^2
\cdot\big[\begin{smallmatrix}z&0\\0&z^{-1}\end{smallmatrix}\big]F\) and \(F\), we obtain
\begin{equation}
\label{ACo}
L_n(z,H)=\tr\Big(F\big[\begin{smallmatrix}z&0\\0&z^{-1}\end{smallmatrix}\big]F\cdot
F\big[\begin{smallmatrix}z&0\\0&z^{-1}\end{smallmatrix}\big]F\cdot
\, \cdots\,\cdot F\big[\begin{smallmatrix}z&0\\0&z^{-1}\end{smallmatrix}\big]F\cdot
F\big[\begin{smallmatrix}z&0\\0&z^{-1}\end{smallmatrix}\big]F\Big)
\end{equation}
According to the definitions \eqref{LM} and \eqref{DPi}, the function in the right hand side of
\eqref{ACo} is the polynomial \(\Pi_n(z,F)\).
\end{proof}

We apply Lemma \ref{Mo} to the \emph{normalized} matrix \(H\) of the form \eqref{Rs}-\eqref{HM}.
In view of \eqref{NoH}, the matrix \(H^{\ast}H\) is of the form
\begin{equation}
\label{sH}
H^{\ast}H=\begin{bmatrix}1&\gamma\\\overline{\gamma}&1\end{bmatrix}, \ \ \text{where} \ \gamma\in\mathbb{C},\ |\gamma|\leq 1.
\end{equation}
Let \(F=\begin{bmatrix}f_{11}&f_{12}\\ f_{21}&f_{22}\end{bmatrix}\)   be the non-negative square root of the matrix \(H^{\ast}H\).
Since \(F\geq0\), the conditions
\begin{equation}
\label{coF}
 f_{11}\geq 0,\ f_{22}\geq 0, \ f_{12}=\overline{f_{21}}
\end{equation}
hold.
In particular,
\begin{equation}
\label{ed}
|f_{12}|=|f_{21}|.
\end{equation}
From \eqref{sH} and from the equality \(F^2=H^{\ast}H\) it follows that
\begin{equation}
\label{eq1}
(f_{11})^2+|f_{12}|^2=1,\ \  |f_{21}|^2+(f_{22})^2=1.
\end{equation}
Since
\(f_{11}\geq0,\,f_{22}\geq0\), from \eqref{ed} and \eqref{eq1} it follows that
\begin{equation}
\label{emd}
f_{11}=f_{22}.
\end{equation}
From \eqref{eq1}, \eqref{emd} and \(f_{12}=\overline{f_{21}}\) it follows that there exist
\(\theta\in[0,\pi/2]\) and \(a\in\mathbb{C},\,|a|=1\) such that \(f_{11}=f_{22}=\cos\theta\),
\(f_{12}=a\sin\theta,\,f_{21}=\sin\theta\,\overline{a}\). Since \(F\geq 0\), the inequality
\(\det F\geq0\) holds. Therefore actually \(\theta\in[0,\pi/4]\). It is evident that such
\(\theta\) and \(a\) are unique.

Thus the following result is obtained:
\begin{lem}
Let \(H=\begin{bmatrix}
 h_{11}&h_{12}\\
 h_{21}&h_{22}
 \end{bmatrix}\) be an arbitrary \(2\times2\) matrix with complex entries which satisfy the
normalizing condition \eqref{NoH}. Let \(F\) be the non-negative square root of the matrix \(H^{\ast}H\).\\[0.5ex]
Then \(F\) is of the form \(F=F_{\theta,a}\), where
 \begin{equation}
 \label{gf}
 F_{\theta,a}=\begin{bmatrix}\cos\theta&a\sin{\theta}\\ \sin\theta\,\overline{a}&\cos\theta\end{bmatrix},
 \end{equation}
 whith \(\theta\in[0,\pi/4],\,a\in\mathbb{C},\,|a|=1\).
\end{lem}
{\ }\\
\noindent
According to Lemma \ref{Mo}, the Laurent polynomials \(L_n(z,H)\) and \(L_n(z,F_{\theta,a})\) coincide:
\begin{equation}
\label{CoPo}
L_n(z,H)=L_n(z,F_{\theta,a}),\ \ n=1,2,3,\,\ldots\,\,.
\end{equation}
Let us relate the matrix \(U_a\) to the number \(a\in\mathbb{C}\):
\begin{equation}
\label{Ua}
U_a=\begin{bmatrix}a&0\\ 0&1\end{bmatrix}.
\end{equation}
If \(|a|=1\), then the matrix \(U_a\) is unitary: \(U_aU_a^{\ast}=U_a^{\ast}U_a=
\big[\begin{smallmatrix}1&0\\ 0&1\end{smallmatrix}\big]\). For \(|a|=1\),
the equalities
\begin{equation}
\label{Ue}
F_{\theta,a}=U_aF_{\theta}U_{a}^{\ast},\ \
U_a\big[\begin{smallmatrix}z&0\\ 0&z^{-1}\end{smallmatrix}\big]U_a^{\ast}=
\big[\begin{smallmatrix}z&0\\ 0&z^{-1}\end{smallmatrix}\big].
\end{equation}
hold, where
\begin{equation}
\label{Fth}
F_{\theta}=\begin{bmatrix}\cos\theta&\sin\theta\\ \sin\theta&\cos\theta\end{bmatrix}.
\end{equation}
Therefore
\begin{equation*}
S(z,F_{\theta,a})=U_aS(z,F_{\theta})U_a^{\ast}, \quad\forall \,a\in\mathbb{C}:\,|a|=1,
\end{equation*}
and for any \(n=1,2,3,\,\ldots\,,\)
\begin{equation}
\label{ure}
(S(z,F_{\theta,a}))^n=U_a(S(z,F_{\theta}))^nU_a^{\ast}, \quad\forall \,a\in\mathbb{C}:\,|a|=1,
\end{equation}
If \(M\) is an arbitrary matrix and \(U\) is an unitary matrix, then
\(\tr UMU^{\ast}=\tr M\). In particular,
 \(\tr U_a(S(z,F_{\theta}))^nU_a^{\ast}=\tr (S(z,F_{\theta}))^n\). Thus
 \begin{equation}
 \label{coip}
 L_n(z,F_{\theta,a})=L_n(z,F_\theta), \quad
 \forall a\in\mathbb{C}:\,|a|=1.
 \end{equation}
 Comparing \eqref{CoPo} and \eqref{coip}, we obtain the following result:
 \begin{thm}
 \label{Rete}
 Let \(H=\begin{bmatrix}
 h_{11}&h_{12}\\
 h_{21}&h_{22}
 \end{bmatrix}\) be an arbitrary \(2\times2\) matrix with complex entries which satisfies the
 normalizing condition \eqref{NoH}. Then there exists an unique \(\theta\in[0,\pi/4]\) such that
 \begin{equation}
 \label{FCc}
 L_n(z,H)=L_n(z,F_{\theta}), \ \  n=1,2,3,\,\ldots \,\,,
 \end{equation}
 where the matrix \(F_{\theta}\) is defined by \eqref{Fth}.
 \end{thm}
 \noindent
 Let as summarize  the above consideration.
 \begin{thm}
 \label{NoF}
 Let \(G=\big[\begin{smallmatrix}g_{11}&g_{12}\\ g_{21}&g_{22}\end{smallmatrix}\big]\) be a \(2\times2\)
 matrix with complex entries. We assume that \(G\) is generic, that in no-one of two columns \(\big[\begin{smallmatrix}g_{11}\\ g_{21}\end{smallmatrix}\big]\) and \(\big[\begin{smallmatrix}g_{12}\\ g_{22}\end{smallmatrix}\big]\)
 vanishes. Let the Laurent polynomial \(L_n(z,G)\) is defined by \eqref{DePi}.\\
Then
\begin{enumerate}
\item
There exists the number \(\theta\in[0,\pi/4]\)
 such that  the Laurent polynomial \(L_n(\,.\,,G)\) generated by the matrix \(G\)
 can be expressed in terms of the Laurent polynomial \(L_n(\,.\,,F_\theta)\)
 generated by the matrix \(F_\theta\):
 \begin{equation}
 \label{NorFo}
 L_n(z,G)\equiv R^nL_n(\rho z,F_\theta),\ \ z\in\mathbb{C},
 \end{equation}
 for every \(n=1,2,3,\,\,\ldots\,\,\),
 where the  matrix \(F_\theta\) is defined by \eqref{Fth}, the numbers \(R\) and \(\rho\) are the same that appears in \eqref{NoRa}.
 \item
  The parameters \(\theta\) is determined
 by the matrix \(G\) uniquely.
 In particular \(\theta\) does not depend on \(n\).
 \item The parameter \(\theta\) takes the value \(\theta=0\) if and only if the columns \(\big[\begin{smallmatrix}g_{11}\\ g_{21}\end{smallmatrix}\big]\) and \(\big[\begin{smallmatrix}g_{12}\\ g_{22}\end{smallmatrix}\big]\) of the matrix \(G\) are orthogonal, that is
 \(g_{11}\overline{g_{12}}+g_{21}\overline{g_{22}}=0\).\\
 The parameter \(\theta\) takes the value \(\theta=\pi/4\) if and only if the columns \(\big[\begin{smallmatrix}g_{11}\\ g_{21}\end{smallmatrix}\big]\) and \(\big[\begin{smallmatrix}g_{12}\\ g_{22}\end{smallmatrix}\big]\) of the matrix \(G\) are proportional, that is the matrix \(G\) is of rank one.
 \end{enumerate}
 \end{thm}
  \section{Properties of the polynomials \(\boldsymbol{L_n(z,F_{\theta})}.\)}
 \label{sec2}
 \begin{thm}
 \label{TaP}
 Let the Laurent polynomial \(L_n(z,F_\theta)\), \(n=1,2,3,\,\,\ldots\,\,\), be defined as
 \begin{subequations}
 \label{Rpt}
 \begin{equation}
 \label{rpt}
 L_n(z,F_\theta)=\tr \,(S(z,F_\theta))^n,
 \end{equation}
 where
 \begin{equation}
 \label{reps}
 S(z,F_\theta)=F_\theta
 \begin{bmatrix}z&0\\0&z^{-1} \end{bmatrix}
 F_\theta,
 \end{equation}
 \end{subequations}
 the matrix \(F_\theta\) is defined by \eqref{Fth}, and \(\theta\in[0,\pi/4]\).

 Then
 \begin{enumerate}
 \item
 The Laurent polynomial \(L_n(z,F_\theta)\) is of the form
 \begin{multline}
 \label{FPi}
\hfill L_n(z,F_\theta)=z^n+z^{-n}+\hspace*{-3.0ex}\sum\limits_{-(n-1)\leq k\leq (n-1)}p_{k,n}(\theta)z^k,\hfill \\  p_{k,n}(\theta)=p_{-k,n}(\theta)\,.
 \end{multline}
\item For \(\theta\in(0,\pi/4)\), the coefficients \(p_{k,n}(\theta)\) vanish if
  \(k\not=n \Mod\) and  are strictly positive if \( k=n \Mod\):
  \begin{subequations}
  \label{StP}
      \begin{align}
      \label{StP1}
       p_{k,n}(\theta)=0,&\quad -(n-1)\leq k \leq n-1,\ k\not=n \Mod,\\
      \label{StP2}
      p_{k,n}(\theta)>0,&\quad -(n-1)\leq k \leq n-1,\ k=n \Mod.
      \end{align}
      \end{subequations}
\item For \(\theta\in[0,\frac{\pi}{4})\), the Laurent polynomial \(L_n(z,F_\theta)\) can be expressed in terms of the Chebyshev polynomial \(T_n,\,T_n(\zeta)=\cos(n\arccos \zeta)\):
\begin{equation}
\label{EChP}
L_n(z,F_\theta)=2(\cos2\theta)^n\cdot T_n\big(\tfrac{z+z^{-1}}{2\cos2\theta}\big).
\end{equation}
\end{enumerate}
\end{thm}
{\ }
\begin{rem}
\label{SpTet}
For \(\theta=0\),
\begin{equation}
\label{StP0}
L_n(z,F_0)=z^{n}+z^{-n}\,.
\end{equation}
 So, all coefficients \(p_{k,n}(0)\) of the Laurent polynomial \(L_n(z,F_0)\) vanish:
 \begin{equation}
 \label{Stp2}
 p_{k,n}(0)=0,\quad -(n-1)\leq k \leq n-1.
 \end{equation}
 For \(\theta=\frac{\pi}{4}\),
 \begin{equation}
 \label{Qua1}
 L_n(z,\tfrac{\pi}{4})=\big(z+z^{-1}\big)^n,
 \end{equation}
 so
 \begin{equation}
 \label{Qua2}
 p_{k,n}(\tfrac{\pi}{4})=0 \ \textup{ if } k\not=n\!\!\!\mod 2,\quad
 p_{k,n}(\tfrac{\pi}{4})=\binom{n}{\frac{n-k}{2}} \ \textup{ if } k=n\!\!\!\mod 2.
 \end{equation}
\end{rem}
\begin{proof}[Proof of Theorem \ref{TaP}] {\ } \\
\textbf{1.} The equalities \eqref{TaP} are the equalities \eqref{DePi} for the matrix \(G=F_\theta\).\\[1.0ex]
\textbf{2.} It is clear that \(S(z,F_\theta)=zP_1(\theta)+P_{-1}(\theta)z^{-1}\), where
\begin{equation}
\label{Is}
P_1(\theta)=\begin{bmatrix}\cos\theta\\ \sin\theta\end{bmatrix}
\begin{bmatrix}\cos\theta&\sin\theta\end{bmatrix},\quad
P_{-1}(\theta)=\begin{bmatrix}\sin\theta\\ \cos\theta\end{bmatrix}
\begin{bmatrix}\sin\theta&\cos\theta\end{bmatrix}.
\end{equation}
Thus
\begin{equation}
\label{spn}
(S(z,F_\theta))^n=\sum\limits_{\boldsymbol\varepsilon}z^{\nu(\boldsymbol{\varepsilon})}
P_{\varepsilon_1}(\theta)\cdot P_{\varepsilon_2}(\theta)\cdot\,\,\cdots\,\,\cdot P_{\varepsilon_n}(\theta),
\end{equation}
the sum in \eqref{spn} runs over all combination\footnote{
There are \(2^n\) such combinations.
}
\(\boldsymbol{\varepsilon}=\varepsilon_1\varepsilon_2\,\,\ldots\,\,\varepsilon_n\) of subscripts
with either \(\varepsilon_k=1\) or \(\varepsilon_k=-1\),
\(\nu(\boldsymbol{\varepsilon})=\varepsilon_1+\varepsilon_2+\,\,\cdots\,\,+\varepsilon_n\).

It is clear that
\begin{equation*}
\nu(\boldsymbol{\varepsilon})=n-2\nu_{-}(\boldsymbol{\varepsilon})=
2\nu_{+}(\boldsymbol{\varepsilon})-n,
\end{equation*}
where
\begin{equation}
\label{nupm}
\nu_{+}(\boldsymbol{\varepsilon})=\#\{k:\,\varepsilon_k=+1\},
\quad
\nu_{-}(\boldsymbol{\varepsilon})=\#\{k:\,\varepsilon_k=-1\}.
\end{equation}
Therefore
\begin{equation}
\label{emt}
\nu(\boldsymbol{\varepsilon})=n \Mod \ \ \forall\,
\boldsymbol{\varepsilon}=(\varepsilon_1,\varepsilon_2,\,\,\cdots\,\,,\varepsilon_n).
\end{equation}
Regrouping summands in \eqref{spn}, we obtain
\begin{equation}
\label{spnr}
(S(z,F_\theta))^n=\sum\limits_{-n\leq k\leq n}\!\!z^{k}\bigg(
\sum\limits_{\boldsymbol{\varepsilon}:\nu(\boldsymbol{\varepsilon})=k}
P_{\varepsilon_1}(\theta)\cdot P_{\varepsilon_2}(\theta)\cdot\,\,\cdots\,\,\cdot P_{\varepsilon_n}(\theta)\bigg).
\end{equation}
Thus the coefficients \(p_{k,n}(\theta)\) of the polynomial \(L_n(z,F_\theta)\), \eqref{FPi}, are:
\begin{multline}
\label{Econ}
p_{k,n}(\theta)=\sum\limits_{\boldsymbol{\varepsilon}:\nu(\boldsymbol{\varepsilon})=k}
\tr\big(P_{\varepsilon_1}(\theta)\cdot P_{\varepsilon_2}(\theta)\cdot\,\,\cdots\,\,\cdot P_{\varepsilon_n}(\theta)\big),\\[-2.0ex]  -(n-1)\leq k\leq n-1,
 \end{multline}
the sum in \eqref{Econ} runs over the set  \(\{\boldsymbol{\varepsilon}:\,\nu(\boldsymbol{\varepsilon})=k\}\).

According to \eqref{emt},  if \(k\not=n\Mod\), then the set \(\{\boldsymbol{\varepsilon}:\,\nu(\boldsymbol{\varepsilon})=k\}\) is empty.
Thus  the sum in \eqref{Econ} vanishes if \(k\not=n\Mod\).
In other words, the condition \eqref{StP1} holds.
If an integer \(k\) satisfies the conditions
\begin{equation}
\label{Ink}
k=n\,\Mod,\quad -(n-1)\leq k\leq(n-1),
\end{equation}
then the set \(\{\boldsymbol{\varepsilon}:\,\nu(\boldsymbol{\varepsilon})=k\}\) is not empty.
The equality \(\nu(\boldsymbol{\varepsilon})=k\) means that
\[\nu_{+}(\boldsymbol{\varepsilon})=\tfrac{n+k}{2},\,\,
\nu_{-}(\boldsymbol{\varepsilon})=\tfrac{n-k}{2}.\]
Moreover if an integer \(k\) satisfies the condition \eqref{Ink}, then
\begin{equation}
\label{Card}
\#\{\boldsymbol{\varepsilon}:\,\nu(\boldsymbol{\varepsilon})=k\}=
\binom{n}{\tfrac{n+k}{2}}=\binom{n}{\tfrac{n-k}{2}}.
\end{equation}
For \(\theta\in(0,\pi/2)\), all entries each of the matrices \(P_{\varepsilon_j}(\theta)\) are
strictly positive. Hence all the entries each of the matrices
\(P_{\varepsilon_1}(\theta)\cdot P_{\varepsilon_2}(\theta)\cdot\,\,\cdots\,\,\cdot P_{\varepsilon_n}(\theta)\) are strictly positive.
All the more \(\tr \big(P_{\varepsilon_1}(\theta)\cdot P_{\varepsilon_2}(\theta)\cdot\,\,\cdots\,\,\cdot P_{\varepsilon_n}(\theta)\big)>0\). Therefore the condition \eqref{StP2} holds.\\[1.0ex]
\textbf{3.} For a \(2\times2\) matrix
\(M=\big[\begin{smallmatrix}m_{11}&m_{12}\\ m_{21}&m_{22}\end{smallmatrix}\big]\),
let \(\lambda_1(M)\) and \(\lambda_2(M)\) be the eigenvalues of \(M\), that is the roots
of the characteristic equation
\(\det\big(\big[\begin{smallmatrix}\lambda&0\\ 0&\lambda\end{smallmatrix}\big]-M\big)=0\).
For any power \(M^n\) of the matrix \(M\), \(n=1,2,3,\,\,\ldots\,\,\), the equalities
\begin{equation*}
\lambda_1(M^n)=(\lambda_1(M))^n,\quad \lambda_2(M^n)=(\lambda_2(M))^n
\end{equation*}
 hold. In particular, the trace \(\tr M^n\) of the matrix \(M^n\) can be expressed
 in terms of the eigenvalues of the matrix \(M\):
 \begin{equation}
 \label{trn}
 \tr M^n=(\lambda_1(M))^n+(\lambda_2(M))^n,\ \ n=1,2,3,\,\,\ldots\,\,.
 \end{equation}
 We apply \eqref{trn} to the matrix \(M=S(z,F_\theta)\). Taking into account \eqref{rpt}, we come
 to the equality
 \begin{subequations}
 \begin{equation}
 \label{ERe}
 L_n(z,F_{\theta})=\big(\lambda_1(S(z,F_\theta))\big)^n+\big(\lambda_2(S(z,F_\theta))\big)^n, \ \
 z\in\mathbb{C}.
  \end{equation}
 The eigenvalues of the matrix \(S(z,F_\theta)\) can be found explicitly:
 \begin{align}
 \lambda_1(S(z,F_\theta))&
 =\tfrac{z+z^{-1}}{2}+\sqrt{\left(\tfrac{z+z^{-1}}{2}\right)^2-\cos^22\theta},
 \notag\\[-1.0ex]
 \label{EoS}
 \\[-1.0ex]
 \lambda_2(S(z,F_\theta))&
 =\tfrac{z+z^{-1}}{2}-\sqrt{\Big(\tfrac{z+z^{-1}}{2}\Big)^2-\cos^22\theta}.\notag
 \end{align}
 \end{subequations}
 The Chebyshev polynomial \(T_n(\zeta)\) can be represented as
 \begin{subequations}
 \label{CPR}
 \begin{gather}
 \label{CPR1}
 T_n(\zeta)=\tfrac{1}{2}\big((\mu_1(\zeta))^n+(\mu_2(\zeta))^n\big),\\
 \intertext{where}
  \label{CPR2}
  \mu_1(\zeta)=\zeta+\sqrt{\zeta^2-1},\quad \mu_2(\zeta)=\zeta-\sqrt{\zeta^2-1}.
 \end{gather}
 \end{subequations}
 Comparing \eqref{EoS} with \eqref{CPR2}, we conclude that
 \begin{multline}
 \label{CoCo}
 \lambda_1(S(z,F_\theta))=\cos2\theta\,\mu_1\big(\tfrac{z+z^{-1}}{2\cos2\theta}\big),
 \quad
 \lambda_2(S(z,F_\theta))=\cos2\theta\,\mu_2\big(\tfrac{z+z^{-1}}{2\cos2\theta}\big),\\
 \theta\in[0,\pi/4),\ \ z\in\mathbb{C}\setminus 0.
 \end{multline}
 Comparing \eqref{ERe}, \eqref{CPR1} and \eqref{CoCo}, we obtain \eqref{EChP}.
 \end{proof}
 \begin{rem}
 \label{lic}
 Since \(2T_n(\frac{z+z^{-1}}{2})=z^n+z^{-n}\), the equality \eqref{StP0} is the special
 case of the equality \eqref{EChP} corresponding to the value \(\theta=0\).

 Since \(T_n(\zeta)=2^{n-1}\zeta^n+o(|\zeta|^n)\) as \(|\zeta|\to\infty\),
 the equality \eqref{Qua1} is the limiting case of the equality \eqref{EChP} corresponding to
the "value" \(\theta=\frac{\pi}{4}-0\).
 \end{rem}
 \begin{rem}
 \label{app} Applying the binomial formula, we derive from \eqref{CPR} that
 \begin{equation}
 \label{aechp}
 T_n({\zeta})=\sum\limits_{0\leq j\leq\frac{n}{2}}\binom{n}{2j}\zeta^{n-2j}(\zeta^2-1)^{j}\,.
 \end{equation}
 Substituting \(\zeta=\frac{z+z^{-1}}{2\cos2\theta}\) into \eqref{aechp} and taking into account
  \eqref{EChP}, we obtain the equality
 \begin{equation}
 \label{arlp}
 L_n(z,F_\theta)=2^{-(n-1)}\sum\limits_{0\leq j\leq\frac{n}{2}}\binom{n}{2j}
 (z+z^{-1})^{n-2j}\big(z^2+z^{-2}+2(1-\cos2\theta)\big)^{j}.
 \end{equation}
 From \eqref{arlp} it is evident that
 \[(-1)^nL_n(-z,F_\theta)=L_n(z,F_\theta).\]
 Hence the condition \eqref{StP1} holds. The condition \eqref{StP2} can be derived from \eqref{arlp}.
 The coefficients \(p_{k,n}(\theta)\) of the Laurent polynomial \(L_n(z,F_\theta)\) majorize the coefficients of the Laurent polynomial
 \(2^{-(n-1)}(z+z^{-1})^n\).
 \end{rem}
 \noindent\textbf{Notation.} Let \(f:\,\mathbb{C}\to\mathbb{C}\) be a mapping and   \(E\subset\mathbb{C}\)
be a subset of \(\mathbb{C}\). By \(f^{[-1]}(E)\) we denote the preimage of the set \(E\) with respect to the mapping~\(f\):
\[f^{[-1]}(E)=\{\zeta\in\mathbb{C}:\,f(\zeta)\in E\}. \]

\begin{lem}
\label{Prcchp} Let \(T_n\) be the Chebyshev polynomial of degree \(n\), \(n=1,2,3,\,\,\ldots\,\,.\) Then
\begin{equation}
\label{Prche}
T_n^{[-1]}\big([-1,1]\big          )=[-1,1].
\end{equation}
\end{lem}
\begin{proof}{\ }\\
\textbf{a.} If \(\zeta\in[-1,1]\), then \(T_n(\zeta)\in[-1,1]\). Thus \(T_n^{[-1]}[-1,1]\supseteq[-1,1].\)\\
\textbf{b.} If \(s\in[-1,1]\), then the equation \(T_n(\zeta)=s\) has \(n\) roots \(\zeta_1(s),\,\ldots\,,\zeta_n(s)\) located within the interval \([-1,1]\). (If \(s\in(-1,1)\) these roots are even different.)
Since the polynomial \(T_n\) is of degree \(n\), the equation \(T_n(t)=s\) has no other roots.
Thus, if \(s\in[-1,1]\), then \(T_n^{[-1]}(\{s\})\subset[-1,1]\).
Thus \(T_n^{[-1]}([-1,1])\subseteq[-1,1]\).
\end{proof}
\noindent
 Let us introduce the mapping \(\Psi_\theta:\,\mathbb{C}\setminus0\to\mathbb{C}\):
 \begin{equation}
 \label{coma}
 \Psi_\theta(z)=\frac{z+\frac{1}{z}}{2\cos2\theta},
 \end{equation}
 \(\theta\in\big[0,\frac{\pi}{4}\big)\) is considered as a parameter. The mapping \(\Psi_\theta\)
 is related to the Joukowski mapping \(Jo:\,\mathbb{C}\setminus0\to\mathbb{C}\):
\[Jo(z)=\frac{z+\frac{1}{z}}{2}.\]
Concerning the Joukowski mapping, see for example \cite[Section I.5, pp.67-68.]{FB}

From properties of the Joukowski mapping we derive the following\\
 \textsf{Properties of the mapping} \(\Psi_\theta\):
\begin{enumerate}
\item
{\ }
\vspace*{-5.0ex}
\begin{equation}
\label{prps}
\Psi_\theta^{[-1]}\big([-1,1]\big)=\mathbb{T}_\theta^{+}\cup\mathbb{T}_\theta^{-},\quad
\Psi_\theta^{[-1]}\big((-1,1)\big)=
\overset{\hspace{-1.0ex}\circ}{\mathbb{T}_\theta^{+}}
\cup\overset{\hspace{-1.0ex}\circ}{\mathbb{T}_\theta^{-}},
\end{equation}
where
\begin{align*}
\mathbb{T}_\theta^{+}&=\{z\in\mathbb{C}:\,|z|=1,\,2\theta\leq\phantom{-}\arg z\leq\pi-2\theta\},\\
\mathbb{T}_\theta^{-}&=\{z\in\mathbb{C}:\,|z|=1,\,2\theta\leq-\arg z\leq\pi-2\theta\},\\
\overset{\hspace{-1.0ex}\circ}{\mathbb{T}_\theta^{+}}&=\{z\in\mathbb{C}:\,|z|=1,\,2\theta<\phantom{-}\arg z<\pi-2\theta\},\\
\overset{\hspace{-1.0ex}\circ}{\mathbb{T}_\theta^{-}}&=\{z\in\mathbb{C}:\,|z|=1,\,2\theta<-\arg z<\pi-2\theta\}.
\end{align*}
\item \(\Psi_\theta\) maps \(\mathbb{T}_\theta^{+}\) onto \([-1,1]\) homeomorphically, and
\(\Psi_\theta^{\prime}(z)\not=0\) for \(z\in\overset{\hspace{-1.0ex}\circ}{\mathbb{T}_\theta^{+}}\).
\item \(\Psi_\theta\) maps \(\mathbb{T}_\theta^{-}\) onto \([-1,1]\) homeomorphically, and
\(\Psi_\theta^{\prime}(z)\not=0\) for \(z\in\overset{\hspace{-1.0ex}\circ}{\mathbb{T}_\theta^{-}}\).
\end{enumerate}
\begin{thm}
\label{RoLP}
For each \(\theta\in\big[0,\frac{\pi}{4}\big)\) and for each \(n=1,2,3,\,\,\ldots\,\,\),
all roots of the Laurent polynomial \(L_{n}(z,F_{\theta})\) are located within the set
\(\overset{\hspace{-1.0ex}\circ}{\mathbb{T}_\theta^{+}}\cup%
\overset{\hspace{-1.0ex}\circ}{\mathbb{T}_\theta^{-}}\) and are simple (i.e. of multiplicity one).
\end{thm}
\begin{proof}
We consider the function \(T_n\big(\tfrac{z+z^{-1}}{2\cos2\theta}\big)\) which appears in
\eqref{EChP} as a composition
\(T_n\circ\Psi_\theta\)  of the Chebyshev polynomial \(T_n\) and the function \(\Psi_\theta\) defined by \eqref{coma}. The roots of the polynomial \(T_n\) form the set \(T_n^{[-1]}(\{0\})\).
Since \(\{0\}\in(-1,1)\) and
\begin{math}
\Psi_\theta^{[-1]}\big((-1,1)\big)=
\overset{\hspace{-1.0ex}\circ}{\mathbb{T}_\theta^{+}}
\cup\overset{\hspace{-1.0ex}\circ}{\mathbb{T}_\theta^{-}},
\end{math}
all roots of the function \(T_n\big(\tfrac{z+z^{-1}}{2\cos2\theta}\big)\) lie within the set
\(\overset{\hspace{-1.0ex}\circ}{\mathbb{T}_\theta^{+}}
\cup\overset{\hspace{-1.0ex}\circ}{\mathbb{T}_\theta^{-}}\). Since all roots of \(T_n\) are simple
and the derivative \(\Psi_\theta^{\prime}(z)\) does not vanish for \(z\in\overset{\hspace{-1.0ex}\circ}{\mathbb{T}_\theta^{+}}
\cup\overset{\hspace{-1.0ex}\circ}{\mathbb{T}_\theta^{-}}\), all roots of the function
\(T_n\big(\tfrac{z+z^{-1}}{2\cos2\theta}\big)\) are simple.
Now Theorem \ref{RoLP} is a consequence of the statement 3 of Theorem \ref{TaP}.
(See the equality \eqref{EChP}.)
\end{proof}
\begin{rem}
\label{Sym}
Since
\begin{equation}
\label{SyRe}
L_{n}(z,F_{\theta})=(-1)^nL_{n}(-z,F_{\theta}),\ \
 L_{n}(\overline{z},F_{\theta})=\overline{L_{n}(z,F_{\theta})}\quad \forall\,z\in\mathbb{C}\setminus0,
 \end{equation}
 the set of all roots of the Laurent polynomial \(L_{n}(\,.\,,F_{\theta})\) is symmetric both with respect to the real axis and with respect to the imaginary axis.
\end{rem}
  \section{The parametrization of the set \(\boldsymbol{\{L_n(z,G)\}_G}\) of Laurent polynomials
  by free parameters}
 \label{sec3}
 \begin{thm}
 \label{Cor1}
 For each \(n=2,3,\,\,\ldots\,\,\), the family of Laurent polynomials
 \(\{L_n(z,G)\}_{G}\), where \(G\) runs over the set of all generic \(2\times2\) matrices with complex
 entries, is a three-parametric family. The representation \eqref{NorFo} is a  parametrization of this family by free parameters \(R,\rho,\theta\)\textup{:}
 \begin{enumerate}
 \item
 Given a generic \(2\times2\) matrix \(G\) with complex entries, then
 for every \(n=1,2,3,\,\,\ldots\,\,\)
 the Laurent polynomial \(L_n(z,G)\) is representable in the form
 \begin{equation}
 \label{norFo}
 L_n(z,G)= R^nL_n(\rho z,F_\theta),\ \ z\in\mathbb{C},
 \end{equation}
 with some \(R\in(0,\infty),\,\rho\in(0,\infty),\theta\in[0,\pi/4]\).
 \item Given a triple \((R,\rho,\theta)\) of numbers which satisfy the condition
 \begin{equation}
 \label{TrCo}
 R\in(0,\infty),\,\rho\in(0,\infty),\,\theta\in[0,\pi/4],
 \end{equation}
  then there exists the generic matrix \(G_{R,\rho,\theta}\) such that the equalities
 \begin{equation}
 \label{nof}
 L_n(z,G_{R,\rho,\theta})=R^nL_n(\rho z,F_\theta)
 \end{equation}
 hold for every \(n=1,2,3,\,\,\ldots\,\,\).
\item If the triples \((R_1,\rho_1,\theta_1)\) and \((R_2,\rho_2,\theta_2)\) satisfy
the condition \eqref{TrCo} and  the functions \(R_1^nL_n(\rho_1 z,F_{\theta_1})\)
and \(R_2^nL_n(\rho_2 z,F_{\theta_2})\) of variable \(z\) coincide for some
\(n\geq2\), then \(R_1=R_2\), \(\rho_1=\rho_2\), and \(\theta_1=\theta_2\).
If the functions \(R_1^nL_1(\rho_1 z,F_{\theta_1})\)
and \(R_2^nL_1(\rho_2 z,F_{\theta_2})\) coincide, then \(R_1=R_2,\,\rho_1=\rho_2\),
but \(\theta_1,\theta_2\) can be arbitrary.
 \end{enumerate}
 \end{thm}
 \begin{proof}{\ }\\
 \textbf{1}. The statement 1 of Theorem \ref{Cor1}  coincides with the statement 1
 of Theorem~\ref{NoF}.\\
 \textbf{2}. Given a triple \((R,\rho,\theta)\) , we define
 \begin{equation}
 r_1=\sqrt{R\cdot\rho},\ \ r_2=\sqrt{R/\rho}, \ \ G_{R,\rho,\theta}=F_\theta
 \begin{bmatrix}
 r_1&0\\
 0&r_2
 \end{bmatrix}
 \cdot
 \end{equation}
 Then
 \begin{align*}
 G_{R,\rho,\theta}
 \begin{bmatrix}
 z&0\\
 0&z^{-1}
 \end{bmatrix}
 G_{R,\rho,\theta}^\ast&=R\,F_\theta
 \begin{bmatrix}
 \rho z&0\\
 0&(\rho z)^{-1}
 \end{bmatrix}
 F_\theta.\\
 \intertext{In other words,}
 S(z,G_{R,\rho,\theta})&=R\cdot S(\rho z,F_{\theta}).\\
 \intertext{Finally}
 \tr \big(S(z,G_{R,\rho,\theta})^n\big)&=R^n\cdot\tr\big(S(\rho z,F_{\theta})^n\big).
 \end{align*}
 Thus the equality \eqref{nof} holds.\\[1.0ex]
 \textbf{3}. We assume that
 \begin{equation}
 \label{CoEq}
 R_1^nL_n(\rho_1 z,F_{\theta_1})=R_2^nL_n(\rho_2 z,F_{\theta_2}) \ \ \forall\,z\in\mathbb{C}\setminus 0
 \end{equation}
 by some \(n\). According to \eqref{FPi}, the equality \eqref{CoEq} implies that
 \begin{gather*}
 R_1^n\big((\rho_1z)^n+(\rho_1z)^{-n}+\hspace*{-3.5ex}\sum\limits_{-(n-1\leq k\leq(n-1))}\hspace*{-3.5ex}p_k(\theta_1)(\rho_1z)^k\big)=\\
= R_2^n\big((\rho_2z)^n+(\rho_2z)^{-n}+\hspace*{-3.5ex}\sum\limits_{-(n-1\leq k\leq(n-1))}\hspace*{-3.5ex}p_k(\theta_2)(\rho_2z)^k\big).
 \end{gather*}
 Comparing the coefficients by the leading terms \(z^n\) and \(z^{-n}\) we see that
 \begin{equation*}
 R_1^n\rho_1^n=R_2^n\rho_2^n,\quad R_1^n\rho_1^{-n}=R_2^n\rho_2^{-n}.
 \end{equation*}
 From these equalities it follows that \(R_1=R_2\) and \(\rho_1=\rho_2\). Now the equality
 \eqref{CoEq} is reduced to the equality
 \begin{equation*}
 L_n(z,F_{\theta_1})= L_n(z,F_{\theta_2})\ \ \forall\,z\in\mathbb{C}\setminus0.
 \end{equation*}
 In particular,
 \begin{equation*}
 L_n(1,F_{\theta_1})= L_n(1,F_{\theta_2}).
 \end{equation*}
 According to the statement 2 of Theorem \ref{TaP}, the value \(L_n(1,F_\theta)\)
 increases strictly monotonically in the interval \(\theta\in[0,\pi/4]\) if \(n\geq2\).
 Therefore \(\theta_1=\theta_2\).

 The Laurent polynomial \(L_1(z,F_\theta)=z+z^{-1}\) does not depend on \(\theta\).
 \end{proof}
\section{Trigonometric polynomials generated by \(\boldsymbol{2\times2}\) matrices}
 \label{sec4}
The formula \eqref{EChP} suggests to relate the family of trigonometric polynomials
\(\taul_{n,\theta}(t)\) to the family of Laurent polynomials \(L_n(z,\theta)\):
\begin{defn}
\label{DeTP}
For \(\theta\in[0,\tfrac{\pi}{4})\) and \(n=1,2,3,\,\,\ldots\,\,\),
we define the function \(\taul_{n,\theta}(t)\) of variable \(t\in\mathbb{C}\):
\begin{equation}
\label{DeTp}
\taul_{n,\theta}(t)\stackrel{\textup{\tiny def}}{=}\tfrac{1}{2(\cos2\theta)^n}\cdot L_n(e^{it},F_\theta),
\end{equation}
where the function \(L_n(z,F_\theta)\) was defined by \eqref{Rpt}.
\end{defn}
\begin{lem}
\label{TryP}
The function \(\taul_{n,\theta}(t)\) is an even trigonometric polynomial of degree \(n\):
\begin{equation}
\label{ETP}
\taul_{n,\theta}(t)=\frac{1}{(\cos2\theta)^n}\cos nt+
\sum\limits_{0\leq k\leq (n-1)}\tau_{k,n}(\theta)\cos kt,
\end{equation}
where the coefficients \(\tau_{k,n}(\theta)\) are related to the coefficients \(p_{k,n}(\theta)\)
of the Laurent polynomial \(L_n(z,F_\theta)\), \eqref{FPi}, by the equalities
\[\tau_{0,n}(\theta)=\frac{p_{0,n}(\theta)}{2(\cos2\theta)^n},
\ \ \tau_{k,n}(\theta)=\frac{p_{k,n}(\theta)}{(\cos2\theta)^n},\ 1\leq k\leq n-1.\]
In particular,
\begin{subequations}
  \label{StT}
      \begin{align}
      \label{StT1}
       \tau_{k,n}(\theta)=0,&\quad -(n-1)\leq k \leq n-1,\ k\not=n \Mod,\\
      \label{StT2}
      \tau_{k,n}(\theta)>0,&\quad -(n-1)\leq k \leq n-1,\ k=n \Mod.
      \end{align}
      \end{subequations}
\end{lem}
\newpage
The following result is an immediate consequence of Theorem \ref{TaP}, statement 3:
\begin{thm}
\label{RChp}
The trigonometric polynomial \(\taul_n(t,\theta)\) and the Chebyshev polynomial \(T_n(\zeta)\)
are related by the equality
\begin{equation}
\label{rChp}
\taul_{n,\theta}(t)=T_n\big(\tfrac{\cos t}{\cos 2\theta}\big), \ \ t\in\mathbb{C},
\ \ \theta\in[0,\tfrac{\pi}{4}).
\end{equation}
\end{thm}

\noindent
 Let \(\Phi_\theta:\,\mathbb{C}\to\mathbb{C}\) be the mapping defined as
\begin{equation}
\label{Psi}
\Phi_\theta(t)=\tfrac{\cos t}{\cos2\theta},
\end{equation}
 where \(\theta\in\big[0,\tfrac{\pi}{4}\big)\) is considered as
a parameter.
\begin{lem}
\label{Phip}
For \(\theta\in\big[0,\frac{\pi}{4}\big)\), the function \(\Phi_\theta\) possesses the
following properties:
\begin{enumerate}
\item
\begin{subequations}
\label{PreI}
\begin{equation}
\label{PreI1}
\Phi_\theta^{[-1]}\big([-1,1]\big)=\mathscr{P},\quad \Phi_\theta^{[-1]}\big((-1,1)\big)=\overset{\circ}{\mathscr{P}},
\end{equation}
 where
\begin{align}
\label{PreI2}
\mathscr{P}&=\!\!\!\!\bigcup\limits_{-\infty<p<\infty}[p\pi+2\theta,(p+1)\pi-2\theta],\\
\label{PreI3}
\overset{\circ}{\mathscr{P}}&=\!\!\!\!\bigcup\limits_{-\infty<p<\infty}(p\pi+2\theta,(p+1)\pi-2\theta),
\end{align}
\end{subequations}
are periodic systems of closed or open intervals respectively.
\item
For each \(p\), the function \(\Phi_\theta\) maps the interval
\([p\pi+2\theta,(p+1)\pi-2\theta]\) onto the interval \([-1,1]\) homeomorphically.
\item For each \(p\),
 \begin{equation}
 \label{NoVan}
 \Phi_\theta^{\prime}(t)\not=0\quad \forall\,t\in(p\pi+2\theta,(p+1)\pi-2\theta).
 \end{equation}
\end{enumerate}
\end{lem}
\begin{proof}{\ }\\
\textbf{1.} Let
\(\mathscr{S}=\{t\in\mathbb{C}:\,\, \cos t\in\mathbb{R}\}\).
Then \(\mathscr{S}\) is the union of the real axis and the countable set of vertical lines:
\[\mathscr{S}=\mathbb{R}\cup\big(\!\!\!\bigcup\limits_{-\infty<q<\infty}\!\!\!t_q+i\mathbb{R}\big),
\quad t_q=\frac{\pi}{2}(1+2q).\]
If \(s\in\mathbb{R}\setminus0\), then \(|\cos(t_q+is)|>1.\) Therefore, \(t\in\mathbb{R}\) if
\(\Phi_\theta(t)\in[-1,1]\).
Thus \eqref{PreI} holds.\\
\textbf{2.} On each interval \([p\pi+2\theta,(p+1)\pi-2\theta]\), the function \(\Phi_\theta(t)\)
behaves strictly monotonically. It decreases if \(p\) is even and increases if \(\pi\) is odd.\\
\textbf{3.} \[
(-1)^{p-1}\Phi_\theta^{\prime}(t)=
\frac{|\sin t|}{\cos 2\theta}>0\,\,\,\forall\,t\in(p\pi+2\theta,(p+1)\pi-2\theta).\]
\end{proof}

\noindent
 According to Theorem \ref{RChp},  the mapping \(\taul_{n,\theta}\) is a composition of the mappings \(\Phi_{\theta}\) and \(T_n\): \(\taul_{n,\theta}=T_n\circ\Phi_{\theta}\). Therefore the following result holds:
\begin{lem}
\label{fupr}
\ \ \ \
For each \(n=1,2,3,\,\,\ldots\,\,\) and \(\theta\in\big[0,\tfrac{\pi}{4}\big)\), the preimage
\(\taul_{n,\theta}^{[-1]}([-1,1])\) of the interval \([-1,1]\) with respect to the mapping
\(\taul_{n,\theta}\) is the system \(\mathscr{P}\) of intervals that appears
in \eqref{PreI2}:
\begin{equation}
\label{PrPr}
\taul_{n,\theta}^{[-1]}([-1,1])=\mathscr{P}.
\end{equation}
\end{lem}

\begin{thm}
\label{Cor}
For each \(n=1,2,3,\,\,\ldots\,\,\) and \(\theta\in\big[0,\tfrac{\pi}{4}\big)\)\textup{:}
\begin{enumerate}
\item All roots of the equation
\begin{equation}
\label{rtp}
\taul_{n,\theta}(t)=0
\end{equation}
 are real and simple. Moreover the roots of the equation \eqref{rtp} are located within the set \(\overset{\circ}{\mathscr{P}}\).
 \item All roots of the equation
 \begin{equation}
 \label{rtp2}
 (\taul_{n,\theta}(t))^2=1
 \end{equation}
 are real.
\end{enumerate}
\end{thm}
\begin{proof} {\ }\\
\textbf{1.} For each \(n=1,2,3,\,\,\ldots\,\,\), \[T_n^{[-1]}(\{0\})\subset(-1,1).\]
Since \(\taul_{n,\theta}=T_n\circ\Phi_{\theta}\),
\[\taul_{n,\theta}^{[-1]}\big(\{0\}\big)\subset\Phi_\theta^{[-1]}\big((-1,1)\big).\]
In view of \eqref{PreI1},
\[\taul_{n,\theta}^{[-1]}\big(\{0\}\big)\subset\overset{\circ}{\mathscr{P}}.\]
In particular,
\[\taul_{n,\theta}^{[-1]}\big(\{0\}\big)\subset\mathbb{R}.\]
All roots of the Chebyshev polynomial are simple. Since
\(\Phi_\theta^{\prime}(t)\not=0\ \ \forall\,t\in\overset{\circ}{\mathscr{P}}\),
all roots of the trigonometric polynomial \(\taul_{n,\theta}=T_n\circ\Phi_{\theta}\) are simple as well.\\
\textbf{2.} The set of roots of the equation \eqref{rtp2} is the set
\[\taul_{n,\theta}^{[-1]}\big(\{-1\}\cup\{+1\}\big)
\subset\taul_{n,\theta}^{[-1]}\big([-1,1]\big)=\mathscr{P}\subset{\mathbb{R}}.\]
\end{proof}
We denote by \(\mathfrak{F}\) the class of all real entire functions \(f(z)\) having the
property (F): \emph{all roots of the equation \(f^2(z)-1=0\) are real.}

Functions of the class \(\mathfrak{F}\) arise in matters:
\begin{enumerate}
\item
Stability theory of linear differential equations
with periodic coefficients, \cite{K};
\item
Spectral theory of of linear differential equations
with periodic coefficients, \cite{MO}.
\item
Approximation theory, \cite{SY}, \cite{Y}.
\end{enumerate}
{\ }\\
Functions belonging to the class \(\mathfrak{F}\) admit a description in terms of
\emph{comb functions}. A comb functions is a function which effects a conformal mapping of the
open upper half-plane onto a comb region. See \cite{EY}.

The comb domain related to the function \(\taul_{n,\theta}(t)\) is shown in Figure 1,
where
\[\cosh h=\frac{1}{\cos 2\theta}.\]
\noindent
\hspace*{-1.5ex}
\begin{center}
{\includegraphics*[scale=0.7]{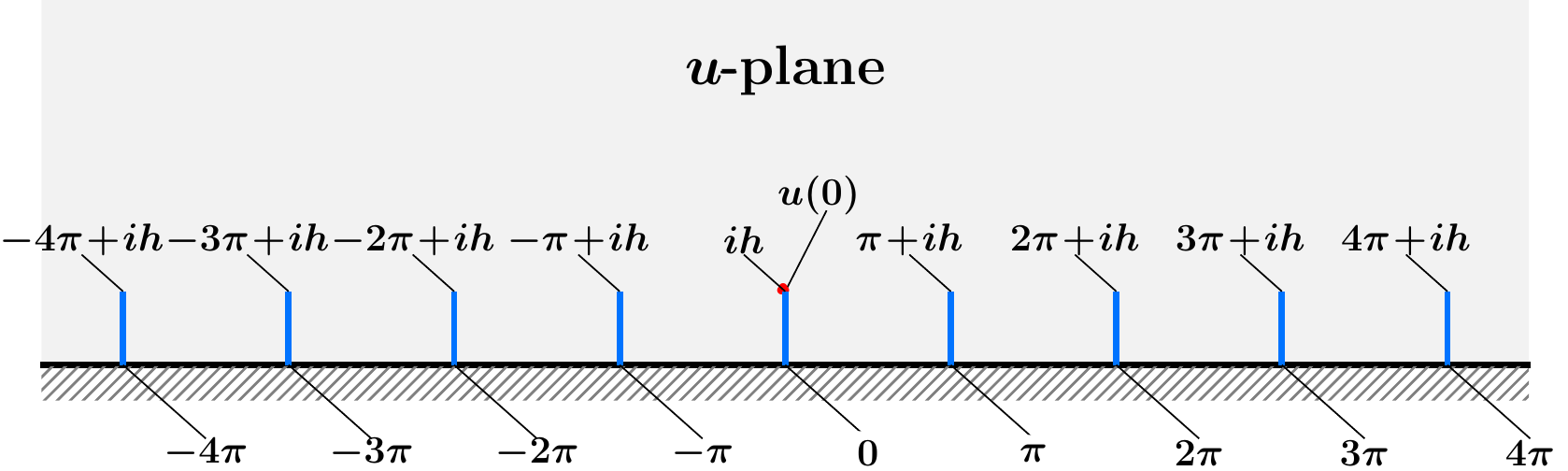}}\\
Figure 1
\end{center}
\normalsize
The function\footnote{Here \(\ln z=0\) for \(z=1\) and
\(\sqrt{z^2-1}>0\) for \(z\in(1,+\infty)\).
}
\[u_\theta(t)=i\ln\Big(\tfrac{\cos t}{\cos2\theta}+\sqrt{(\tfrac{\cos t}{\cos2\theta})^2-1}\,\Big)\]
effects the conformal mapping of the upper half-plane \(\{t\in\mathbb{C}:\,\im t>0\}\) onto the comb domain shown in Fig.1.
The normalizing conditions are:
\[u_\theta(0)=ih, \quad \lim\limits_{t\to i\infty} t^{-1}u_\theta(t)=1.\]
Figure 2 illustrates the boundary correspondence by the mapping \(t\to u_\theta(t)\).
\noindent
\hspace*{-1.5ex}
\begin{center}{\includegraphics*[scale=0.7]{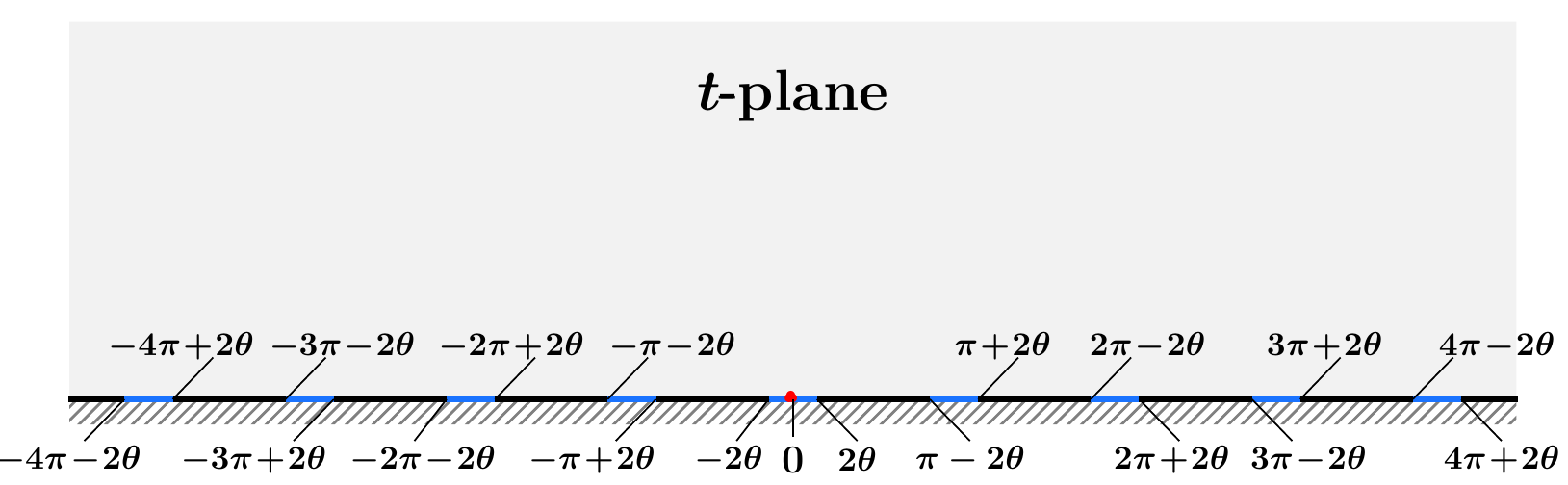}}\\
Figure 2
\end{center}
The function \(\taul_{n,\theta}(t)\) is representable in the form
\[\taul_{n,\theta}(t)=\cos n\,u_\theta(t).\]


\begin{thebibliography}{1111}
\bibitem[EY]{EY} A.\,Eremenko, P.Yuditskii.  \textit{Comb Functions.} Contemporary Mathematics, vol.
\textbf{578}, 2012, 99-118.
\bibitem[FB]{FB} E.Freitag,\,R.Busam. \textit{Complex Analysis.} Springer-Verlag,
Berlin\(\cdot\)Hedelberg, 2005.
\bibitem[K]{K} M.G.Krein. \textit{The basic properties of the theory of \(\lambda\)-zones
of stability of a canonical systems of linear differential
equations with periodic coefficients.}(Russian).  In memory of Aleksandr Aleksandrovich Andronov, pp. 413-498. Izdat. Akad. Nauk SSSR, Moscow, 1955.\\ Reprinted in: M.G.Krein. \textit{Selected Works.} Vol.3,
pp.139-257. Institute of Mathmatics Ukrinian Acad. of Science, Kyiv, 1997.\\ Englich transl.in: M.G.Krein.
\textit{Topics in Differential and Integral Equations and Operator Theory.} Operator Theory: Advances
and applications. \textbf{OT 7}. Springer, Basel, 1983.
pp.1-105.
\bibitem[MO]{MO} V.A.\,Marchenko,\,I.V.Ostrovskii. \textit{A characterization of the spectrum of the Hill operator.} Math.\,Sbornik, \textbf{97}:4 (1975), 540-606. (In Russian).\\
     English translation: Math. USSR-Sb. 26 (1975), no.\,4, 493-554 (1977).
\bibitem[SY]{SY} M.Sodin,\,P.Yuditskii. {Functions that deviate least from zero on closed subsets of the real axis.} (Russian). Algebra i Analis, \textbf{4}:2 (1992), 1-61.\\
     English translation in St. Petersburg Math. J. \textbf{4}:2 (1993), 209-241.

\bibitem[Y]{Y}
P. Yuditskii. \textit{A special case of de Brange's theorem on the inverse monodromy problem.}
Integr. Equ. Oper. Theory, \textbf{39} (2001), 229-252.
\end{thebibliography}
\end{document}